\documentclass[10pt]{article}

\usepackage{hyperref}
\usepackage{xcolor}
\usepackage{amssymb}
\usepackage{amsfonts}
\usepackage{amsthm}
\usepackage{amsmath}
\usepackage{float}
\usepackage{bigints}
\usepackage{fullpage}
\usepackage{mathtools}
\usepackage[utf8]{inputenc} 
\usepackage{cancel}
\usepackage{verbatim}

\newtheorem{Lemma}{Lemma}
\newtheorem{Sublemma}{Sublemma}
\newtheorem{Theorem}{Theorem}

\newtheorem{Definition}{Definition}

\newtheorem{Corollary}{Corollary}
\newtheorem{Remark}{Remark}

\numberwithin{Subcase}{Case}

\DeclareMathOperator{\inter}{int}

\DeclareMathOperator{\x}{\mathbf{x}}

\DeclareMathOperator{\z}{\mathbf{z}}
\DeclareMathOperator{\oo}{\mathbf{o}}
\DeclareMathOperator{\K}{\mathbf{K}}

\DeclareMathOperator{\E}{\mathbb{E}}
\DeclareMathOperator{\B}{\mathbf{B}}
\DeclareMathOperator{\C}{\mathbf{C}}

\DeclareMathOperator{\Q}{\mathbf{Q}}

\DeclareMathOperator{\vol}{vol}

\usepackage[pdftex]{graphicx}
\graphicspath{{C:/Users/Sam/Desktop/PaperFigures/}}
\DeclareGraphicsExtensions{.pdf}

\title{Minimizing the mean projections of finite $\rho$-separable packings
\footnote{Keywords: totally separable packing, translative packing, density, $\rho$-separable packing, convex body, volume, mean projection.}
\footnote{2010 Mathematics Subject Classification: 52C17, 05B40, 11H31, and 52C45.}}

\author{K\'{a}roly Bezdek\thanks{Partially supported by a Natural Sciences and 
Engineering Research Council of Canada Discovery Grant.}
and Zsolt L\'angi\thanks{Partially supported by the National Research, Development and Innovation Office, NKFI, K-119670.}}

\begin{document}

\maketitle

\begin{abstract}
A packing of translates of a convex body in the $d$-dimensional Euclidean space $\E^d$ is said to be totally separable if any two packing elements can be separated by a hyperplane of $\mathbb{E}^{d}$ disjoint from the interior of every packing element. We call the packing $\cal P$ of translates of a centrally symmetric convex body $\C$ in $\E^d$ a $\rho$-separable packing for given $\rho\geq 1$ if in every ball concentric to a packing element of $\cal P$ having radius $\rho$ (measured in the norm generated by $\C$) the corresponding sub-packing of $\cal P$ is totally separable. The main result of this paper is the following theorem. Consider the convex hull $\Q$ of $n$ non-overlapping translates of an arbitrary centrally symmetric convex body $\C$ forming a $\rho$-separable packing in $\E^d$ with $n$ being sufficiently large for given $\rho\geq 1$. If $\Q$ has minimal mean $i$-dimensional projection for given $i$ with $1\leq i<d$, then $\Q$ is approximately a $d$-dimensional ball. This extends a theorem of K. B\"or\"oczky Jr. [Monatsh. Math. \textbf{118} (1994), 41--54] from translative packings to $\rho$-separable translative packings for $\rho\geq 1$.
\end{abstract}

\section{Introduction}\label{sec:intro}
We denote the $d$-dimensional Euclidean space by $\E^d$. Let $\B^d$ denote the unit ball centered at the origin $\oo$ in $\E^d$. A {\it $d$-dimensional convex body} $\C$ is a compact convex subset of $\E^d$ with non-empty interior $\inter \C$. (If $d=2$, then $\C$ is said to be a {\it convex domain}.) If $\C = -\C$, where $-\C=\{-x: x\in \C\}$, then $\C$ is said to be {\it $\oo$-symmetric} and a translate $\mathbf{c}+\C$ of $\C$ is called centrally symmetric with center $\mathbf{c}$.

The starting point as well as the main motivation for writing this paper is the following elegant theorem of B\"or\"oczky Jr. \cite{Bo}: Consider the convex hull $\Q$ of $n$ non-overlapping translates of an arbitrary convex body $\C$ in $\E^d$ with $n$ being sufficiently large. If $\Q$ has minimal mean $i$-dimensional projection for given $i$ with $1\leq i<d$, then $\Q$ is approximately a $d$-dimensional ball. In this paper, our main goal is to prove an extension of this theorem to $\rho$-separable translative packings of convex bodies in $\E^d$. Next, we define the concept of $\rho$-separable translative packings and then state our main result.

A packing of translates of a convex domain $\C$ in $\E^2$ is said to be {\it totally separable} if any two packing elements can be separated by a line of $\mathbb{E}^{2}$ disjoint from the interior of every packing element. This notion was introduced by G. Fejes T\'{o}th and L. Fejes T\'{o}th \cite{FTFT73} . We can define a totally separable packing of translates of a $d$-dimensional convex body $\C$ in a similar way by requiring any two packing elements to be separated by a hyperplane in $\mathbb{E}^{d}$ disjoint from the interior of every packing element \cite{BeKh, BeSzSz}. 

\begin{Definition}
Let $\mathbf{C}$ be an $\oo$-symmetric convex body of  $\mathbb{E}^d$. Furthermore, let $\|\cdot\|_{\mathbf{C}}$ denote the norm generated by $\mathbf{C}$, i.e., let $\|\mathbf{x}\|_{\mathbf{C}}:=\inf\{\lambda\ |\ \mathbf{x}\in\lambda \mathbf{C}\}$ for any $\mathbf{x}\in \mathbb{E}^d$. Now, let $\rho\ge 1$. We say that the packing 
$${\cal P}_{{\rm sep}}:=\{\mathbf{c}_i+\mathbf{C}\ |\ i\in I \ {\rm with}\ \| \mathbf{c}_j-\mathbf{c}_k\|_{\mathbf{C}}\ge 2 \ {\rm for\ all}\ j\neq k\in I\}$$ 
of (finitely or infinitely many) non-overlapping translates of $\mathbf{C}$ with centers $\{\mathbf{c}_i\ |\ i\in I\}$ is a {\rm $\rho$-separable packing} in $\mathbb{E}^d$ if for each $i\in I$ the finite packing $\{\mathbf{c}_j+\mathbf{C}\ |\ \mathbf{c}_j+\mathbf{C}\subseteq \mathbf{c}_i+\rho\mathbf{C}\}$ is a totally separable packing (in $\mathbf{c}_i+\rho\mathbf{C}$). Finally, let $\delta_{{\rm sep}}(\rho, \mathbf{C})$ denote the largest density of all $\rho$-separable translative packings of $\mathbf{C}$ in $\mathbb{E}^d$, i.e., let
$$\delta_{{\rm sep}}(\rho, \mathbf{C}):=\sup_{{\cal P}_{\rm sep}}\left(\limsup_{\lambda\to+\infty}\frac{\sum_{\mathbf{c}_i+\mathbf{C}\subset\mathbf{W}_{\lambda}^d}{\rm vol}_d(\mathbf{c}_i+\mathbf{C})}{{\rm vol}_d(\mathbf{W}_{\lambda}^d)}\right)\ , $$
where $\mathbf{W}_{\lambda}^d$ denotes the $d$-dimensional cube of edge length $2\lambda$ centered at $\mathbf{o}$ in $\mathbb{E}^d$ having edges parallel to the coordinate axes of $\mathbb{E}^d$ and ${\rm vol}_d(\cdot)$ refers to the $d$-dimensional volume of the corresponding set in $\E^d$.
\end{Definition}

\begin{Remark}
Let $\delta(\mathbf{C})$ (resp., $\delta_{{\rm sep}}(\mathbf{C})$) denote the supremum of the upper densities of all translative packings (resp., totally separable translative packings) of
the $\oo$-symmetric convex body $\mathbf{C}$ in $\mathbb{E}^d$. Clearly, $\delta_{{\rm sep}}(\mathbf{C})\leq\delta_{{\rm sep}}(\rho, \mathbf{C})\leq \delta(\mathbf{C})$  for all $\rho\geq 1$. Furthermore, if $1\le \rho< 3$, then any $\rho$-separable translative packing of $\mathbf{C}$ in $\mathbb{E}^d$ is simply a translative packing of $\C$ and therefore, $\delta_{{\rm sep}}(\rho, \mathbf{C})=\delta(\mathbf{C})$. 
\end{Remark}

Recall that the mean $i$-dimensional projection $M_i(\C)$ ($i=1,2,\ldots,d-1$) of the convex body $\C$  in $\E^d$, can be expressed (\cite{Sch}) with the help of mixed volume via the formula 
\[
M_i(\C) = \frac{\kappa_i}{\kappa_d} V(\overbrace{\C,\ldots,\C}^i,\overbrace{\B^d,\ldots,\B^d}^{d-i}),
\]
where $\kappa_d$ is the volume of $\B^d$ in $\E^d$. Note that $M_i(\B^d) = \kappa_i$, and the surface area of $\C$ is $S(\C)= \frac{d \kappa_{d}}{\kappa_{d-1}} M_{d-1}(\C)$ and in particular, $S(\B^d)=d\kappa_d$. Set $M_d(\C):={\rm vol}_d(\C)$. Finally, let $R(\C)$ (resp., $r(\C)$) denote the circumradius (resp., inradius) of the convex body $\C$ in $\E^d$, which is the radius of the smallest (resp., largest) ball that contains (resp., is contained in) $\C$. Our main result is the following.

\begin{Theorem}\label{main-result}
Let $d \geq 2$, $1\leq i\leq d-1$, $\rho\geq 1$, and let $\Q$ be the convex hull of the $\rho$-separable packing of $n$ translates of the $\oo$-symmetric convex body $\C$ in $\E^d$ such that $M_i(\Q)$ is minimal and
$n \geq \frac{4^d d^{4d}}{\delta_{{\rm sep}}(\rho,\C)^{d-1}} \cdot \left( \rho \frac{R(\C)}{r(\C)}\right)^d$. Then
\begin{equation}\label{main}
\frac{r(\Q)}{R(\Q)} \geq 1 - \frac{\omega}{n^{\frac{2}{d(d+3)}}},
\end{equation}
for $\omega = \lambda(d) \left( \frac{\rho R(\C)}{r(\C)} \right)^{\frac{2}{d+3}}$, where $\lambda(d)$ depends only on the dimension $d$. In addition, 
\[
M_i(\Q) = \left( 1+ \frac{\sigma}{n^{\frac{1}{d}}} \right) M_i(\B^d) \left( \frac{{\rm vol}_d(\C)}{\delta_{{\rm sep}}(\rho,\C) \kappa_d} \right)^{\frac{i}{d}} \cdot n^{\frac{i}{d}},
\]
where $- \frac{2.25 R(\C) \rho d i}{r(\C) \delta_{\rm sep}(\rho,\C)} \leq \sigma \leq \frac{ 2.1 R(\C) \rho i}{r(\C) \delta_{\rm sep}(\rho,\C)}$.
\end{Theorem}

\begin{Remark}
It is worth restating Theorem~\ref{main-result} as follows: Consider the convex hull $\Q$ of $n$ non-overlapping translates of an arbitrary $\oo$-symmetric convex body $\C$ forming a $\rho$-separable packing in $\E^d$ with $n$ being sufficiently large. If $\Q$ has minimal mean $i$-dimensional projection for given $i$ with $1\leq i<d$, then $\Q$ is approximately a $d$-dimensional ball.
\end{Remark}

%As every totally separable translative packing of an $\oo$-symmetric convex body $\C$ is a $\rho$-separable translative packing of $\C$ for any $\rho\geq 1$ therefore
%Theorem~\ref{main-result} implies the following statement in a straightforward way.

%\begin{Corollary}
%Consider the convex hull $\Q$ of $n$ non-overlapping translates of an arbitrary $\oo$-symmetric convex body $\C$ forming a totally separable packing in $\E^d$ with {\color{red}$n$ being} sufficiently large. If $\Q$ has minimal mean $i$-dimensional projection for given $i$ with $1\leq i<d$, then $\Q$ is approximately a $d$-dimensional ball.
%\end{Corollary}

\begin{Remark}
The nature of the analogue question on minimizing $M_d(\Q)={\rm vol}_d(\Q)$ is very different. Namely, recall that Betke and Henk \cite{BH} proved L. Fejes T\'oth's sausage conjecture for $d\ge 42$ according to which the smallest volume of the convex hull of $n$ non-overlapping unit balls in $\E^d$ is obtained when the $n$ unit balls form a sausage, that is, a linear packing (see also \cite{BHW} and \cite{BHW-2}). As linear packings of unit balls are $\rho$-separable therefore the above theorem of Betke and Henk applies to $\rho$-separable packings of unit balls in $\E^d$ for all $\rho\geq 1$ and $d\geq 42$. On the other hand, the problem of minimizing the volume of the convex hull of $n$ unit balls forming a $\rho$-separable packing in $\E^d$ remains an interesting open problem for $\rho\geq 1$ and $2\leq d< 42$. Last but not least, the problem of minimizing $M_d(\Q)$ for $\oo$-symmetric convex bodies $\C$ different from a ball in $\E^d$ seems to be wide open for $\rho\geq 1$ and $d\geq 2$.
\end{Remark}

\begin{Remark}
Let $d \geq 2$, $1\leq i\leq d-1$, $n>1$, and let $\C$ be a given $\oo$-symmetric convex body in $\E^d$. Furthermore, let $\Q$ be the convex hull of the totally separable packing of $n$ translates of $\C$ in $\E^d$ such that $M_i(\Q)$ is minimal. Then it is natural to ask for the limit shape of $\Q$ as $n\to+\infty$, that is, to ask for an analogue of Theorem~\ref{main-result} within the family of totally separable translative packings of $\C$ in $\E^d$. This would require some new ideas besides the ones used in the following proof of Theorem~\ref{main-result}.
\end{Remark}

In the rest of the paper by adopting the method of B\"or\"oczky Jr. \cite{Bo} and making the necessary modifications, we give a proof of Theorem~\ref{main-result}.

\section{Basic properties of finite $\rho$-separable translative packings}

The following statement is the $\rho$-separable analogue of the Lemma in \cite{B02} (see also Theorem 3.1 in \cite{BHW}).

\begin{Lemma}\label{R-separable}
Let $\{\mathbf{c}_i+\mathbf{C}\ |\ 1\le i\le n\}$ be an arbitrary $\rho$-separable packing of $n$ translates of the $\oo$-symmetric convex body $\mathbf{C}$ in $\mathbb{E}^d$ with $\rho\ge 1$, $n\ge 1$, and $d\ge 2$. Then
$$\frac{n{\rm vol}_d(\mathbf{C}) }{{\rm vol}_d\left(\cup_{i=1}^n  \mathbf{c}_i+2\rho\mathbf{C}\right)}\le\delta_{{\rm sep}}(\rho, \mathbf{C})\ .$$ 
\end{Lemma}

\begin{proof} We use the method of the proof of the Lemma in \cite{B02} (resp., Theorem 3.1 in \cite{BHW}) with proper modifications. The details are as follows. 
Assume that the claim is not true. Then there is an $\epsilon>0$ such that
\begin{equation}\label{R1}
{\rm vol}_d\left(\cup_{i=1}^n  \mathbf{c}_i+2\rho\mathbf{C}\right)=\frac{n{\rm vol}_d(\mathbf{C}) }{\delta_{{\rm sep}}(\rho, \mathbf{C})}-\epsilon
\end{equation}
Let $C_n=\{\mathbf{c}_i\ |\ i=1,\dots ,n\}$ and let $\Lambda$ be a packing lattice of $C_n+2\rho\mathbf{C}=\cup_{i=1}^n  \mathbf{c}_i+2\rho\mathbf{C}$ such that $C_n+2\rho\mathbf{C}$ is contained in a fundamental parallelotope of $\Lambda$ say, in $\mathbf{P}$, which is symmetric about the origin. Recall that for each $\lambda>0$, $\mathbf{W}_{\lambda}^d$ denotes the $d$-dimensional cube of edge length $2\lambda$ centered at the origin $\mathbf{o}$ in $\mathbb{E}^d$ having edges parallel to the coordinate axes of $\mathbb{E}^d$. Clearly, there is a constant $\mu>0$ depending on $\mathbf{P}$ only, such that for each $\lambda>0$ there is a subset $L_{\lambda}$ of $\Lambda$ with 
\begin{equation}\label{R2}
\mathbf{W}_{\lambda}^d\subseteq L_{\lambda}+\mathbf{P}\ {\rm and}\ L_{\lambda}+2\mathbf{P}\subseteq\mathbf{W}_{\lambda + \mu}^d\ .
\end{equation}
%Moreover, let ${\cal P}_m(\mathbf{B}^d)$ denote the family of all $R$-separable packings of $m>1$ unit balls in $\mathbb{E}^d$. 
The definition of $\delta_{{\rm sep}}(\rho, \mathbf{C})$ implies that for each $\lambda>0$ there exists a $\rho$-separable packing of $m(\lambda)$ translates of $\mathbf{C}$ in 
$\mathbb{E}^d$ with centers at the points of $C(\lambda)$ such that
$$C(\lambda)+\mathbf{C}\subset \mathbf{W}_{\lambda}^d$$
and
$$\lim_{\lambda\to+\infty}\frac{m(\lambda){\rm vol}_d(\mathbf{C})}{{\rm vol}_d( \mathbf{W}_{\lambda}^d)}=\delta_{{\rm sep}}(\rho, \mathbf{C})\ .$$
As $\lim_{\lambda\to+\infty}\frac{{\rm vol}_d(\mathbf{W}_{\lambda+\mu}^d)}{{\rm vol}_d(\mathbf{W}_{\lambda}^d)}=1$ therefore there exist $\xi>0$ and a $\rho$-separable packing
of $m(\xi)$ translates of $\mathbf{C}$ in $\mathbb{E}^d$ with centers at the points of $C(\xi)$ and with $C(\xi)+\mathbf{C}\subset \mathbf{W}_{\xi}^d$ such that
\begin{equation}\label{R3}
\frac{{\rm vol}_d(\mathbf{P})\delta_{{\rm sep}}(\rho, \mathbf{C})}{{\rm vol}_d(\mathbf{P})+\epsilon}<\frac{m(\xi){\rm vol}_d(\mathbf{C})}{{\rm vol}_d( \mathbf{W}_{\xi+\mu}^d)}
\ {\rm and}\    \frac{n{\rm vol}_d(\mathbf{C})}{{\rm vol}_d(\mathbf{P})+\epsilon}<\frac{n{\rm vol}_d(\mathbf{C}){\rm card}(L_{\xi})}{{\rm vol}_d(\mathbf{W}_{\xi+\mu}^d)}\ ,
\end{equation}
where ${\rm card}(\cdot)$ refers to the cardinality of the given set.
Now, for each $\mathbf{x}\in\mathbf{P}$ we define a $\rho$-separable packing of $\overline{m}(\mathbf{x})$ translates of $\mathbf{C}$ in $\mathbb{E}^d$ with centers at the points of
$$\overline{C}(\mathbf{x}):=\{\mathbf{x}+L_{\xi}+C_n\}\cup\{\mathbf{y}\in C(\xi)\ |\ \mathbf{y}\notin \mathbf{x}+L_{\xi}+C_n+{\rm int}(2\rho\mathbf{C})\}\ .  $$
Clearly, (\ref{R2}) implies that $\overline{C}(\mathbf{x})+\mathbf{C}\subset\mathbf{W}_{\xi+\mu}^d$.
Now, in order to evaluate $\int_{\mathbf{x}\in\mathbf{P}}\overline{m}(\mathbf{x})d\mathbf{x}$, we introduce the function $\chi_{\mathbf{y}}$ for each $\mathbf{y}\in C(\xi)$ defined as follows: $\chi_{\mathbf{y}}(\mathbf{x})=1$ if $\mathbf{y}\notin \mathbf{x}+L_{\xi}+C_n+{\rm int}(2\rho\mathbf{C})$ and $\chi_{\mathbf{y}}(\mathbf{x})=0$ for any other $\mathbf{x}\in\mathbf{P}$. Based on the origin symmetric $\mathbf{P}$ it is easy to see that for any  $\mathbf{y}\in C(\xi)$ one has $\int_{\mathbf{x}\in\mathbf{P}}
\chi_{\mathbf{y}}(\mathbf{x})d\mathbf{x}={\rm vol}_d(\mathbf{P})-{\rm vol}_d(C_n+2\rho\mathbf{C})$. Thus, it follows in a straightforward way that

$$\int_{\mathbf{x}\in\mathbf{P}}\overline{m}(\mathbf{x})d\mathbf{x}=\int_{\mathbf{x}\in\mathbf{P}}\big(n{\rm card}(L_{\xi})+\sum_{\mathbf{y}\in C(\xi)}\chi_{\mathbf{y}}(\mathbf{x})\big)d\mathbf{x}=n{\rm vol}_d(\mathbf{P}){\rm card}(L_{\xi})+m(\xi)\big({\rm vol}_d(\mathbf{P})-{\rm vol}_d(C_n+2\rho\mathbf{C})\big)\ .$$ 
Hence, there is a point $\mathbf{p}\in\mathbf{P}$ with
$$\overline{m}(\mathbf{p})\ge m(\xi)\left(1-\frac{{\rm vol}_d(C_n+2\rho\mathbf{C})}{{\rm vol}_d(\mathbf{P})}\right)+n{\rm card}(L_{\xi})$$
and so
\begin{equation}\label{R4}
\frac{\overline{m}(\mathbf{p}){\rm vol}_d(\mathbf{C})}{{\rm vol}_d(\mathbf{W}_{\xi+\mu}^d)}\ge \frac{m(\xi){\rm vol}_d(\mathbf{C})}{{\rm vol}_d(\mathbf{W}_{\xi+\mu}^d)}\left(1-\frac{{\rm vol}_d(C_n+2\rho\mathbf{C})}{{\rm vol}_d(\mathbf{P})}\right)+\frac{n{\rm vol}_d(\mathbf{C}){\rm card}(L_{\xi})}{{\rm vol}_d(\mathbf{W}_{\xi+\mu}^d)}\ .
\end{equation}
Now, (\ref{R1}) implies in a straightforward way that
\begin{equation}\label{R5}
\frac{{\rm vol}_d(\mathbf{P}) \delta_{{\rm sep}}(\rho, \mathbf{C}) }{{\rm vol}_d(\mathbf{P})+\epsilon}\left(1-\frac{{\rm vol}_d(C_n+2\rho\mathbf{C})}{{\rm vol}_d(\mathbf{P})}\right)+\frac{n{\rm vol}_d(\mathbf{C})}{{\rm vol}_d(\mathbf{P})+\epsilon}= \delta_{{\rm sep}}(\rho, \mathbf{C})
\end{equation}
Thus, (\ref{R3}), (\ref{R4}), and (\ref{R5}) yield that
$$\frac{\overline{m}(\mathbf{p}){\rm vol}_d(\mathbf{C})}{{\rm vol}_d(\mathbf{W}_{\xi+\mu}^d)}>\delta_{{\rm sep}}(\rho, \mathbf{C})\ .$$
As $\overline{C}(\mathbf{p})+\mathbf{C}\subset\mathbf{W}_{\xi+\mu}^d$ this contradicts the definition of $\delta_{{\rm sep}}(\rho, \mathbf{C})$, finishing the proof of Lemma~\ref{R-separable}.
\end{proof}

\begin{Definition}
Let $d\geq 2$, $\rho\geq 1$, and let $\K$ (resp., $\C$) be a convex body (resp., an $\oo$-symmetric convex body) in $\E^d$. Then let $\nu_{\C}(\rho ,\K)$ denote the largest $n$ with the property that there exists a $\rho$-separable packing $\{\mathbf{c}_i+\mathbf{C}\ |\ 1\le i\le n\}$ such that $\{\mathbf{c}_i\ |\ 1\le i\le n\}\subset\K$. 
\end{Definition}

\begin{Lemma}\label{prop:Boroczky}
Let $d\geq 2$, $\rho\geq 1$, and let $\K$ (resp., $\C$) be a convex body (resp., an $\oo$-symmetric convex body) in $\E^d$. Then
\[
\left( 1+ \frac{2 \rho R(\C)}{r(\K)} \right)^{-d} \frac{{\rm vol}_d(\C) \nu_{\C}(\rho ,\K)}{\delta_{\rm sep}(\rho ,\C)} \leq {\rm vol}_d(\K) \leq \frac{{\rm vol}_d(\C) \nu_{\C}(\rho ,\K)}{\delta_{\rm sep}(\rho ,\C)}.
\]
\end{Lemma}

\begin{proof}

Observe that Lemma~\ref{R-separable} and the containments $\K+2 \rho \C \subseteq \left( 1 + \frac{2 \rho R(\C)}{r(\K)} \right) \K$ yield the lower bound immediately.

We prove the upper bound.
Let $0 < \varepsilon < \delta_{\rm sep}(\rho ,\C)$. By the definition of $\delta_{\rm sep}(\rho ,\C)$, if $\lambda$ is sufficiently large, then
there is a $\rho$-separable packing $\{\mathbf{c}_i+\mathbf{C}\ |\ 1\le i\le n\}$ such that $C_n:=\{\mathbf{c}_i\ |\ 1\le i\le n\}\subset\mathbf{W}_{\lambda}^d$ and 
\begin{equation}\label{direct-start}
\frac{n{\rm vol}_d(\C)}{{\rm vol}_d(\mathbf{W}_{\lambda}^d)} \geq  \delta_{\rm sep}(\rho ,\C) - \varepsilon .
\end{equation}

\begin{Sublemma}
If $\mathbf{X}$ and $\mathbf{Y}$ are convex bodies in $\E^d$ and $\C$ is an $\oo$-symmetric convex body in $\E^d$, then
\begin{equation}\label{almost-trivial}
\nu_{\C}(\rho ,\mathbf{Y}) \geq \frac{{\rm vol}_d(\mathbf{Y}) \nu_{\C}(\rho ,\mathbf{X})}{{\rm vol}_d(\mathbf{X}-\mathbf{Y})}.
\end{equation} 
\end{Sublemma}

\begin{proof}
Indeed, consider any finite point set $X:=\{\x_1,\dots ,\x_N\} \subset \mathbf{X}$. Observe that the following are equivalent for a positive integer $k$:
\begin{itemize}
\item $k$ is the maximum number a point of $\mathbf{X}-\mathbf{Y}$ is covered by the sets $\x_i -\mathbf{Y}$, $\x_i \in X$,
\item $k$ is the maximum number such that ${\rm card}((\z + \mathbf{Y}) \cap X)=k$ for some point $\z\in \mathbf{X}-\mathbf{Y}$.
\end{itemize}

Thus, $N\vol_d(\mathbf{Y}) \leq {\rm card}( (\z + \mathbf{Y}) \cap X) \vol_d(\mathbf{X}-\mathbf{Y})$ for some $\z \in \mathbf{X}-\mathbf{Y}$. Hence, if  $\{\mathbf{x}_i+\mathbf{C}\ |\ 1\le i\le N\}$ is an arbitrary $\rho$-separable packing with $X=\{\x_1,\dots ,\x_N\} \subset \mathbf{X}$, then
\[
\nu_{\C}(\rho ,\mathbf{Y}) \geq {\rm card}((\z + \mathbf{Y}) \cap X) \geq \frac{\vol_d(\mathbf{Y}) N}{\vol_d(\mathbf{X}-\mathbf{Y})},
\]
which implies (\ref{almost-trivial}).
\end{proof}

Applying (\ref{almost-trivial}) to $\mathbf{X}=\mathbf{W}_{\lambda}^d$ and $\mathbf{Y}=-\K$ and using (\ref{direct-start}), we obtain
\[
\nu_{\C}(\rho ,\K) \geq \frac{n\vol_d(\K)}{\vol_d(\mathbf{W}_{\lambda}^d + \K)} 
\geq \frac{\vol_d(\K)}{\vol_d(\mathbf{W}_{\lambda + R(\K)}^d )} \cdot \frac{\vol_d(\mathbf{W}_{\lambda}^d) ( \delta_{\rm sep}(\rho ,\C) - \varepsilon)}{\vol_d(\C)},
\]
which finishes the proof of Lemma~\ref{prop:Boroczky}. 
\end{proof}

\begin{Definition}
Let $d\geq 2$, $n\geq 1$, $\rho\geq 1$, and let $\C$ be an $\oo$-symmetric convex body in $\E^d$. Then let $R_{\C}(\rho , n)$ be the smallest radius $R>0$ with the property that $\nu_{\C}(\rho ,R\B^d)\geq n$.
\end{Definition}

Clearly, for any $\varepsilon > 0$ we have $\nu_{\C}(\rho , (R_{\C}(\rho ,n)-\varepsilon)\B^d) < n$, and thus, by Lemma~\ref{prop:Boroczky} (for $\K=R_{\C}(\rho ,n)\B^d$), we obtain

\begin{Corollary}
Let $d\geq 2$, $n\geq 1$, $\rho\geq 1$, and let $\C$ be an $\oo$-symmetric convex body in $\E^d$. Then
\begin{equation}\label{eq:2}
R_{\C}(\rho ,n)^d \leq \frac{{\rm vol}_d(\C) n}{\delta_{\rm sep}(\rho ,\C) \kappa_d} \leq \left( R_{\C}(\rho ,n) + 2 \rho R(\C) \right)^d . 
\end{equation}
\end{Corollary}

\begin{Lemma}\label{lem:Boroczky}
Let $n \geq \frac{4^d \delta_{\rm sep}(\rho ,\C) \rho^d R(\C)^d}{r(\C)^d}$ and $i=1,2,\ldots,d-1$. Then for $R=R_{\C}(\rho ,n)$,
\[
M_i((R+ \rho R(\C))\B^d) \leq M_i(\B^d) \left( \frac{{\rm vol}_d(\C) n}{\delta_{\rm sep}(\rho ,\C) \kappa_d} \right)^{\frac{i}{d}} \left( 1 + \frac{2\delta_{\rm sep}(\rho ,\C)^{\frac{1}{d}} \rho R(\C)}{r(\C)} \cdot \frac{1}{n^{\frac{1}{d}}} \right)^i .
\]
\end{Lemma}

\begin{proof}
Set $t=R + 2 \rho R(\C)$. Then the first inequality in (\ref{eq:2}) yields that
\[
R+\rho R(\C) \leq \frac{t-\rho R(\C)}{t-2 \rho R(\C)} \left( \frac{\vol_d(\C) n}{\delta_{\rm sep}(\rho ,\C) \kappa_d} \right)^{\frac{1}{d}}.
\]
Thus, by the second inequality in (\ref{eq:2}) and the condition that $n \geq \frac{4^d \delta_{\rm sep}(\rho ,\C) \rho^d R(\C)^d}{r(\C)^d} \geq \frac{4^d \delta_{\rm sep}(\rho ,\C) \rho^d R(\C)^d \kappa_d}{\vol_d(\C)}$, we obtain that
\[
\frac{t-\rho R(\C)}{t-2\rho R(\C)} = 1+ \left( \frac{t}{\rho R(\C)} - 2 \right)^{-1} \leq 1 + \frac{2 \delta_{\rm sep}(\rho ,\C)^{\frac{1}{d}} \rho R(\C) \kappa_d^{\frac{1}{d}}}{\vol_d(\C)^{\frac{1}{d}}} \cdot \frac{1}{n^{\frac{1}{d}}}\leq 1 + \frac{2 \delta_{\rm sep}(\rho ,\C)^{\frac{1}{d}} \rho R(\C)}{r(\C)} \cdot \frac{1}{n^{\frac{1}{d}}}.
\]
\end{proof}

\section{Proof of Theorem~\ref{main-result}}

%Since $\vol_d(\Q) \geq \vol_d(\C) n$, it follows that $r(\Q)$ is of order $n^{\frac{1}{d}}$. Thus, by Lemma~\ref{prop:Boroczky}, we have $\vol_d(\Q) \sim \frac{\vol_d(\C)}{\delta_{\rm sep}(\rho ,\C)} n$. 
%Furthermore, since $\vol_d(\Q)$ is essentially given, the special case of the Alexendrov-Fenchel inequality we use (the isoperimetric inequality for mean $i$th width), Lemma~\ref{lem:Boroczky} and the minimality of $M_i(\Q)$ implies that $M_i(\Q)$ is of order $n^{\frac{i}{d}}$.
%We carry out the exact computations, and denote by $\B$ the unit $d$-dimensional ball centered at the origin $\oo$.
In the proof that follows we are going to use the following special case of the Alexandrov-Fenchel inequality (\cite{Sch}): if $\K$ is a convex body in $\E^d$ satisfying $M_i(\K)\leq M_i(r\B^d)$ for given $1\leq i<d$ and $r>0$, then
\begin{equation}\label{A-F-inequality}
M_j(\K)\leq M_j(r\B^d) 
\end{equation}
holds for all $j$ with $i<j\leq d$. In particular, this statement for $j=d$ can be restated as follows: if $\K$ is a convex body in $\E^d$ satisfying $M_d(\K)= M_d(r\B^d)$ for given $d\geq 2$ and $r>0$, then $M_i(\K)\geq M_i(r\B^d)$ holds for all $i$ with $1\leq i<d$.

Let $d \geq 2$, $1\leq i\leq d-1$, $\rho\geq 1$, and let $\Q$ be the convex hull of the $\rho$-separable packing of $n$ translates of the $\oo$-symmetric convex body $\C$ in $\E^d$ such that $M_i(\Q)$ is minimal and
\begin{equation}\label{assumption}
n \geq \frac{4^d d^{4d}}{\delta_{{\rm sep}}(\rho,\C)^{d-1}} \cdot \left( \rho \frac{R(\C)}{r(\C)}\right)^d.
\end{equation}

By the minimality of $M_i(\Q)$ we have that
\begin{equation}\label{min}
M_i(\Q) \leq M_i(R \B^d+\C) \leq M_i((R+\rho R(\C))\B^d)
\end{equation}
with $R=R_{\C}(\rho ,n)$.
Note that (\ref{min}) and Lemma~\ref{lem:Boroczky} imply that
\[
M_i(\Q) \leq \left( 1 + \frac{2\delta_{\rm sep}(\rho ,\C)^{\frac{1}{d}} \rho R(\C)}{r(\C)} \cdot \frac{1}{n^{\frac{1}{d}}} \right)^{i} M_i(\B^d) \left( \frac{{\rm vol}_d(\C)}{\delta_{{\rm sep}}(\rho,\C) \kappa_d} \right)^{\frac{i}{d}} \cdot n^{\frac{i}{d}} .
\]
We examine the function $x \mapsto (1+x)^i$, where, by (\ref{assumption}), we have $x \leq x_0 = \frac{1}{2d^4}$. The convexity of this function implies that
$(1+x)^i \leq 1 + i (1+x_0)^{i-1} x$. Thus, from the inequality $\left( 1 + \frac{1}{2d^4} \right)^{d-1} \leq \frac{33}{32} < 1.05$, where $d \geq 2$, the upper bound for $M_i(\Q)$ in Theorem~\ref{main-result} follows.

On the other hand, in order to prove the lower bound for $M_i(\Q)$ in Theorem~\ref{main-result}, we start with the observation that (\ref{A-F-inequality}) (based on (\ref{min})), (\ref{assumption}), and Lemma~\ref{lem:Boroczky} yield that
\begin{equation}\label{eq:S(Q)}
S(\Q) \leq S((R+\rho R(\C))\B^d) \leq d \kappa_d \left( \frac{n\vol_d(\C) }{\delta_{\rm sep}(\rho ,\C) \kappa_d} \right)^{\frac{d-1}{d}} \left( 1 + \frac{2\delta_{\rm sep}(\rho ,\C)^{\frac{1}{d}} \rho R(\C)}{r(\C)} \cdot \frac{1}{n^{\frac{1}{d}}} \right)^{d-1} .
\end{equation}

Thus, (\ref{eq:S(Q)}) together with the inequalities $S(\Q) r(\Q)\geq \vol_d(\Q)$ (cf. \cite{Osserman}) and $\vol_d(\Q) \geq n\vol_d(\C) $ yield
\begin{equation}\label{inradius-1}
r(\Q) \geq \left( 1 + \frac{2\delta_{\rm sep}(\rho ,\C)^{\frac{1}{d}} \rho R(\C)}{r(\C)} \cdot \frac{1}{n^{\frac{1}{d}}} \right)^{-(d-1)} \frac{\vol_d(\C)^{\frac{1}{d}} \delta_{\rm sep}(\rho ,\C)^{\frac{d-1}{d}}}{d \kappa_d^{\frac{1}{d}}} \cdot n^{\frac{1}{d}}.
\end{equation}
Applying the assumption (\ref{assumption}) and $\vol_d(\C) \geq \kappa_d r(\C)^d$ to (\ref{inradius-1}), we get that
\begin{equation}\label{inradius-2}
r(\Q) \geq \left( 1 + \frac{1}{2 d^{4}} \right)^{-(d-1)}  \frac{\delta_{\rm sep}(\rho ,\C)^{\frac{d-1}{d}} r(\C)}{d} n^{\frac{1}{d}} \geq\frac{4d^3}{(1+\frac{1}{2d^4})^{d-1}}R(\C)\geq 31 R(\C).
\end{equation}

Let $\mathbf{P}$ denote the convex hull of the centers of the translates of $\C$ in $\Q$.
Then, (\ref{inradius-2}) implies
\begin{equation}\label{inradius-3}
r(\mathbf{P}) \geq r(\Q) - R(\C) \geq \frac{30}{31} r(\Q)\geq \frac{8 \delta_{\rm sep}(\rho ,\C)^{\frac{d-1}{d}} r(\C)}{9d} \cdot n^{\frac{1}{d}}.
\end{equation}
Hence, by (\ref{inradius-3}) and Lemma~\ref{prop:Boroczky},
\begin{equation}\label{eq:V(Q)}
\vol_d(\Q) \geq \vol_d(\mathbf{P}) \geq \left( 1+ \frac{9 d \rho R(\C)}{4 \delta_{\rm sep}(\rho ,\C)^{\frac{d-1}{d}} r(\C)} \cdot \frac{1}{n^{\frac{1}{d}}} \right)^{-d} \cdot \frac{n \vol_d(\C)}{\delta_{\rm sep}(\rho ,\C)},
\end{equation}
which implies in a straightforward way that
\begin{equation}\label{eq:4}
\vol_d(\Q) \geq \left( 1 + \frac{9 d \rho R(\C)}{ 4 \delta_{\rm sep}(\rho ,\C) r(\C)} \cdot \frac{1}{n^{\frac{1}{d}}} \right)^{-d} \cdot \frac{n \vol_d(\C)}{\delta_{\rm sep}(\rho ,\C)}.
\end{equation}
Note that (\ref{A-F-inequality}) (see the restated version for $j=d$) implies that $M_i(Q) \geq \left( \frac{\vol_d(Q)}{\kappa_d} \right)^{\frac{i}{d}} \kappa_i$.
Then, replacing $\vol_d(Q)$ by the right-hand side of (\ref{eq:4}), and using the convexity of the function $x \mapsto (1+x)^{-i}$ for $x > -1$ yields the lower bound for $M_i(\Q)$in Theorem~\ref{main-result}.

Finally, we prove the statement about the spherical shape of $\Q$, that is, the inequality (\ref{main}).
As in \cite{Bo}, let
\[
\theta(d) = \frac{1}{2^{\frac{d+3}{2}} \sqrt{2\pi} \sqrt{d} (d-1)(d+3)} \min \left\{ \frac{3}{\pi^2 d (d+2) 2^d} , \frac{16}{(d\pi)^{\frac{d-1}{4}}}  \right\} .
\]
Using the inequality $\frac{\kappa_{d-1}}{\kappa_d} \geq \sqrt{\frac{d}{2\pi}}$ (cf. \cite{BGW}) and (6) of \cite{GSch}, we obtain
\[
\left( \frac{S(\Q)}{S(\B^d)} \right)^{d} \left( \frac{\vol_d(\B^d)}{\vol_d(\Q)} \right)^{d-1} - 1 \geq \theta(d) \cdot \left( 1 - \frac{r(\Q)}{R(\Q)} \right)^{\frac{d+3}{2}}
\]
(see also (5) of \cite{Bo}).
Substituting (\ref{eq:S(Q)}) and (\ref{eq:V(Q)}) into this inequality, we obtain
\[
\left( 1 + \frac{2 \delta_{\rm sep}(\rho ,\C)^{\frac{1}{d}} \rho R(\C)}{r(\C)} \cdot \frac{1}{n^{\frac{1}{d}}} \right)^{d(d-1)} \left( 1 + \frac{9 d \rho R(\C)}{4 \delta_{\rm sep}(\rho ,\C)^{\frac{d-1}{d}} r(\C)} \cdot \frac{1}{n^{\frac{1}{d}}} \right)^{d(d-1)} \geq \left( \frac{S(\Q)}{S(\B^d)} \right)^{d} \left( \frac{\vol_d(\B^d)}{\vol_d(\Q)} \right)^{d-1} .
\]
By the assumptions $d \geq 2$ and (\ref{assumption}), it follows that
\begin{equation}\label{eq:final}                                                                                                                    
4d^2 (d-1)  \frac{\rho R(\C)}{\delta_{\rm sep}(\rho ,\C) r(\C)} \cdot \frac{1}{n^{\frac{1}{d}}} \geq \theta(d) \left( 1 - \frac{r(\Q)}{R(\Q)} \right)^{\frac{d+3}{2}} .
\end{equation}
Note that by \cite{PSz}, $\frac{1}{\delta_{\rm sep}(\rho ,\C)} \leq \frac{2^{\frac{3d}{2}} \cdot \sqrt{\left( \begin{array}{c} \frac{d(d+1)}{2} \\ d \end{array} \right) }}{ (d+1)^{\frac{d}{2}} \pi^{\frac{d}{2}} \Gamma\left( \frac{d}{2} + 1 \right)}$. This and (\ref{eq:final}) implies (\ref{main}), finishing the proof of Theorem~\ref{main-result}.

\vspace{1cm}

\medskip

\noindent K\'aroly Bezdek \\
\small{Department of Mathematics and Statistics, University of Calgary, Calgary, Canada}\\
\small{Department of Mathematics, University of Pannonia, Veszpr\'em, Hungary}\\
\small{\texttt{bezdek@math.ucalgary.ca}}

\bigskip

\noindent and

\bigskip

\noindent Zsolt L\'angi \\
\small{MTA-BME Morphodynamics Research Group and Department of Geometry}\\ 
\small{Budapest University of Technology and Economics, Budapest, Hungary}\\
\small{\texttt{zlangi@math.bme.hu}}


\begin{thebibliography}{GGM}

\bibitem{BGW} U. Betke, P. Grittzmann and J. Wills, Slices of L. Fejes T\'oth's Sausage Conjecture, {\it Mathematika} \textbf{29} (1982), 194-201.

\bibitem{BHW}
U.~Betke, M.~Henk, and J.~M.~Wills, Finite and infinite packings, {\it J. Reine Angew. Math.} \textbf{453} (1994), 165--191.


\bibitem{BHW-2}
U.~Betke, M.~Henk, and J.~M.~Wills, Sausages are good packings, {\it Discrete Comput. Geom.} \textbf{13/3-4} (1995), 297--311.


\bibitem{BH}
U.~Betke and M.~Henk, Finite packings of spheres, {\it Discrete Comput. Geom.} \textbf{19/2} (1998), 197--227.

\bibitem{B02} 
K.~Bezdek, On the maximum number of touching pairs in a finite packing of translates of a convex body, {\it J. Combin. Theory Ser. A} \textbf{98/1} (2002), 192--200.

\bibitem{BeKh}
K. Bezdek and M. A. Khan, Contact numbers for sphere packings, arXiv:1601.00145v2 [math.MG], 22 January, 2016. 

\bibitem{BeSzSz}
K. Bezdek, B. Szalkai and I. Szalkai, On contact numbers of totally separable unit sphere packings, {\it Discrete Math.} \textbf{339/2} (2016), 668--676. 

\bibitem{Bo} K. B\"or\"oczky, Jr., Mean projections and finite packings of convex bodies, {\it Monatsh. Math.} \textbf{118} (1994), 41--54.

\bibitem{FTFT73}
G.~Fejes T\'oth and L.~Fejes T\'oth, On totally separable domains, {\it Acta Math. Acad. Sci. Hungar.} \textbf{24} (1973), 229--232.

\bibitem{GSch} H. Groemer and R. Schneider, Stability estimates for some geometric inequalities, {\it Bull. London Math. Soc.} \textbf{23} (1991), 67-74.

\bibitem{Osserman} R. Osserman, Bonnesen-type isoperimetric inequalities, {\it Amer. Math. Monthly} \textbf{86} (1979), 1-29.

\bibitem{PSz} A. Pelczynski and S.J. Szarek, On parallelepipeds of minimal volume containing a convex symmetric body, {\it Math. Proc. Camb. Phil. Soc.} \textbf{109} (1991), 125-148.

\bibitem{Sch}
R.~Schneider, Convex bodies: the Brunn-Minkowski theory, {\it Encyclopedia of Mathematics and its Applications} \textbf{44}, Cambridge University Press, Cambridge, 1993.


\end{thebibliography}
\end{document}